\documentclass[12pt,a4paper]{amsart}
\usepackage[top=35mm, bottom=32mm, left=28mm, right=28mm]{geometry}
\usepackage{mathptmx}
\usepackage{mathrsfs}
\usepackage{amscd}
\usepackage{verbatim}
\usepackage{url}
\usepackage[all]{xy}
\usepackage{color}
\usepackage[colorlinks=true,citecolor=blue]{hyperref}
\usepackage{amsmath}
\usepackage{bm}

\theoremstyle{plain}

\newtheorem{thm}{Theorem}[section]
\newtheorem{lem}[thm]{Lemma}
\newtheorem{prop}[thm]{Proposition}
\newtheorem{cor}[thm]{Corollary}

\theoremstyle{definition}
\newtheorem{defn}[thm]{Definition}
\newtheorem{rem}[thm]{Remark}

\newcommand{\Z}{{\mathbb{Z}}}
\newcommand{\N}{\mathbb{N}}

\newcommand{\ep}{\varepsilon}
\newcommand{\ra}{\rightarrow}
\newcommand{\lra}{\longrightarrow}

\DeclareMathOperator{\diam}{diam}

\def \RP {{\bf RP}}
\def \AP {{\bf AP}}
\newcommand{\E}{{\mathbb{E}}}

\def \d {\delta}

\def \X {\mathcal{X}}

\begin{document}
\title[A refined saturation theorem for polynomials and applications]{A refined saturation theorem for polynomials and applications}
\author{Xiangdong Ye and Jiaqi Yu}
\address{School of Mathematical Sciences, University of Science and Technology of China, Hefei, Anhui, 230026, P.R. China}

\email{yexd@ustc.edu.cn}
\email{yjq2020@mail.ustc.edu.cn}

\keywords{saturation theorem, integral polynomials, regionally proximal relation}
\renewcommand{\thefootnote}{}
\footnotetext{2020 {\it Mathematics Subject Classification.} Primary: 37B05}
\thanks{This research is supported by NNSF of China 12031019.}

\begin{abstract} For a dynamical system $(X,T)$, $d\in\N$ and distinct non-constant integral polynomials $p_1,\ldots, p_d$ vanishing at $0$, the notion of regionally proximal relation along $C=\{p_1,\ldots,p_d\}$ (denoted by $\RP_C^{[d]}(X,T)$) is introduced.

It turns out that for a minimal system, $\RP_C^{[d]}(X,T)=\Delta$ implies that $X$ is an almost one-to-one extension of $X_k$ for some $k\in\N$ only depending on a set of finite polynomials associated with $C$ and has zero entropy, where $X_k$ is the maximal $k$-step pro-nilfactor of $X$.

Particularly, when $C$ is a collection of linear polynomials, it is proved that $\RP_C^{[d]}(X,T)=\Delta$ implies $(X,T)$ is a $d$-step pro-nilsystem, which answers negatively a conjecture in \cite{5p}.
The results are obtained by proving a refined saturation theorem for polynomials.
\end{abstract}

\maketitle

\section{Introduction}
In this section we give the motivation of our study and state the main results of the paper.

\subsection{The motivation}\
To study the convergence of the multiple ergodic averages, the notion of characteristic factors was introduced by Furstenberg and Weiss in \cite{F2}. By the results of Host-Kra \cite{HK05} and Ziegler \cite{Zie}, the characteristic factor of the multiple ergodic averages $\frac{1}{N}\sum_{n=0}^{N-1}f_{1}(T^nx)\cdots f_d(T^{dn}x)$ is some $(d-1)$-step pro-nilsystem.

\medskip
The motivation of this paper comes from the consideration of characteristic factors in topological dynamics. In the topological setting, the corresponding notion was first introduced by Glasner in \cite{G94}. To be precise, let $(X,T)$ be a dynamical system and $d\in\mathbb{N}$, set $\tau_d=T\times T^2\times \cdots \times T^d$. $(Y,T)$ is a {\em topological characteristic factor (TCF for short) of order $d$} if there exists a dense $G_{\delta}$ set $\Omega$ of $X$ such that for each $x\in\Omega$ the orbit closure $L_x^d=\overline{O}(x^{(d)},\tau_d)$ is $\pi^{(d)}=:\pi\times\cdots \times \pi$ ($d$ times) saturated, i.e., $(\pi^{(d)})^{-1}\pi^{(d)}(L_x^d)=L_x^d$, 
where $x^{(d)}=(x,\ldots,x)$ ($d$ times) and $\pi:X\to Y$ is the factor map.

\medskip

Glasner, Hunag, Shao, Weiss and Ye in \cite{5p} obtained the following theorem.
Here we state it in the equivalence form.

\medskip

\noindent{\bf Theorem GHSWY:}
{\it Let $(X,T)$ be a minimal system and $\pi:X\rightarrow X_\infty$ be the factor map, where $X_\infty$ is the maximal $\infty$-step
pro-nilfactor of $X$. Then there is a dense $G_\delta$ set $\Omega$ of $X$ such that for any $d\in\N$ and $x\in \Omega$,
$$(\pi^{(d)})^{-1}\pi^{(d)}(x^{(d)})\subset L_x^d=\overline{O}(x^{(d)},\tau_d).$$
Particularly, if $\pi$ is open then $X_\infty$ is the TCF of order $d$ for any $d\in\N$.
}

\medskip
Now let $p_1,\ldots, p_d$ be integral polynomials vanishing at $0$. It is easy to extend the notion of TCF of order $d$ to general polynomials
$p_1,\ldots, p_d$. A natural question is how to generalize the Theorem GHSWY to general polynomials.
In the recent nice work by Qiu \cite{Qiu}, the author proved a version of the saturation theorem for polynomials, which allows him to give a complete answer to a well known question concerning the density of polynomial orbits in a totally minimal systems.
The result of Qiu was strengthened in \cite{HSY-22-2} by Huang, Shao and Ye, and now we state the equivalence form.

\medskip
\noindent{\bf Theorem HSY:}
{\it Let $(X,T)$ be a minimal system, and $\pi:X\rightarrow X_\infty$ be the factor map. Let $d\in\N$ and $p_1, p_2,\ldots, p_d$ be distinct non-constant integral polynomials vanishing at $0$. Then there is a dense $G_\delta$ set $\Omega$ of $X$ such that for any 
$x\in \Omega$,
$$(\pi^{(d)})^{-1}\pi^{(d)}(x^{(d)})\subset L_x^{C}=:\overline{\{(T^{p_1(n)}x, \ldots, T^{p_d(n)}x): n\in\Z\}},$$
where $C=\{p_1,\ldots,p_d\}$.
Particularly, if $\pi$ is open then $X_\infty$ is the TCF
of order $d$ for C, 
i.e., $(\pi^{(d)})^{-1}\pi^{(d)}(L_x^{C})=L_x^{C}.$
}

\subsection{Main results}
One of the main results of the paper is a refinement of Theorem HSY. 

\medskip
\noindent{\bf Theorem A:}
{\it Let $(X,T)$ be a minimal system, $d\in \N$ and  $p_1, p_2, \ldots, p_d$ be distinct non-constant integral polynomials vanishing at $0$. Then there are $k\in \N$ (depending only on the polynomials) and a dense $G_\delta$ set $\Omega$ of $X$ such that for any $x\in \Omega$,
\begin{equation}\label{satu-fi}
(\pi_k^{(d)})^{-1}\pi_k^{(d)}(x^{(d)})\subset L_x^{C}=:\overline{\{(T^{p_1(n)}x, \ldots, T^{p_d(n)}x): n\in\Z\}},
\end{equation}
where $C=\{p_1,\ldots,p_d\}$, $X_k$ is the maximal $k$-step pro-nilfactor and $\pi_k:X\ra X_k$ is the factor map.
Particularly, if $\pi_k$ is open then $X_k$ is the TCF
of order $d$ for $C$. 
}

\medskip


We remark that when $p_1,\ldots,p_d$ are linear polynomials, we can choose $k=d-1$.

Moreover, when $d=1$, we can choose $k=1$ and $X_1$ can be replaced by $X_{rat}$, where $X_{rat}$ is a factor of $X_1$ associated with all
(topological) rational eigenvalues. This implies that if $(X,T)$ is totally minimal, then there is a dense $G_\delta$ set $\Omega$ of $X$ such that for each $x\in \Omega$, $\{T^{p_1(n)}x:n\in \Z\}$ is dense in $X$ (as in this case $X_{rat}$ is trivial), which gives another approach to the result in \cite{5p, Qiu} for a single polynomial.


\medskip

Theorem A is proved by using Theorem HSY, an ergodic argument to deal with the distal case, and a method to put them together.

\medskip
Theorem A has applications to the regionally proximal relation along polynomials we now introduce, 
which is a natural generalization of the regionally proximal relation along linear polynomials (denoted by $\AP^{[d]}(X,T)$) introduced in \cite{4p}.
Precisely, let $(X,T)$ be a topological system,
and $C=\{p_1,\ldots,p_d\}$ be a set of distinct non-constant integral polynomials vanishing at $0$. We say $(x,y)\in X\times X$ is {\it regionally proximal along
$C$} (denoted by $(x,y)\in \RP_C^{[d]}(X,T)$) if for a given $\ep>0$ and a neighborhood $U\times V$ of $(x,y)$, there are a pair $(x',y')\in U\times V$ and $n\in\Z$ such that
$\rho(T^{p_i(n)}x',T^{p_i(n)}y')<\ep$ for any $1\le i\le d$. In a certain sense, this definition was inspired by the study of the multiple
ergodic averages
$$\frac{1}{N}\sum_{n=0}^{N-1}f_1(T^{p_1(n)}x)\cdots f_d(T^{p_d(n)}x).$$

\medskip
Using Theorem A and the properties of $\RP^{[d]}_C$ we obtain the other main result of the paper.

\medskip
\noindent{\bf Theorem B:}
{\it Let $(X,T)$ be a minimal system, $d\in \N$ and  $C=\{p_1, p_2, \ldots, p_d\}$ be distinct non-constant integral polynomials vanishing at $0$. Then
$\RP_C^{[d]}(X,T)=\Delta$ implies that $X$ is an almost one to one extension of $X_k$ for some $k\in \N$ only depending on a set of finite  polynomials derived from $C$, and has zero entropy.
Particularly, if $(X,T)$ is distal then it is a $k$-step pro-nilsystem.}

\medskip
As a corollary we get that if $\AP^{[d]}(X,T)=\Delta$ then $(X,T)$ is a $d$-step pro-nilsystem, which gives a negative answer to a conjecture in \cite[Conjecture 4]{5p}.

\subsection{Organization of the paper}
We organize the paper as follows. In Section 1, we give the motivation of this paper and state the main results. In Section 2 we introduce some necessary notions and some known facts to be used in the paper. In Section 3, we prove the refined saturation theorem. In Section 4, we give some applications to regionally proximal relation along polynomials and give a negative answer to a conjecture in \cite{5p}. Moreover, we ask
several open questions at the end of the paper.

\medskip

\noindent {\bf Acknowledgement}: We thank Jiahao Qiu, Song Shao and Hui Xu for useful discussions.

\section{Preliminary}

In this section we give some necessary notions and some known facts used in the paper.

\subsection{Topological dynamical systems}

\subsubsection{}
By a {\em topological dynamical system} (for short t.d.s.) we mean a pair $(X,T)$, where
$X$ is a  compact metric space $X$ with a metric $\rho$
and $T:X\to X$ is a
homeomorphism. Let $(X, T)$ be a t.d.s. and $x\in X$. Then $O(x,T)=\{T^nx: n\in \Z\}$ denotes the
{\em orbit} of $x$. A subset $A\subseteq X$ is called {\em invariant} (or {$T$-invariant}) if $TA= A$. When $Y\subseteq X$ is a closed and
invariant subset of the system $(X, T)$, we say that the system
$(Y, T|_Y)$ is a {\em subsystem} of $(X, T)$. Usually we will omit the subscript, and denote $(Y, T|_Y)$ by $(Y,T)$.
If $(X, T)$ and $(Y, S)$ are two t.d.s., their {\em product system} is the
system $(X \times Y, T\times S)$.

\medskip

When there are more than one t.d.s. involved, usually we should use different symbols to denote different transformations on different spaces, for example, $(X,T), (Y,S), (Z,H)$ etc. But when no confusing, it is convenient to use only one symbol $T$ for all transformations in all t.d.s. involved,  for example, $(X,T), (Y,T), (Z,T)$ etc.
In this paper, we {\em  use the same symbol $T$ for the transformations in all t.d.s.}

\subsubsection{}
Let $X, Y$ be compact metric spaces and $\phi: X \to Y$ be a map.  For $n \geq 2$, let $\phi^{(n)}=\phi\times \cdots \times \phi: X^n\rightarrow Y^n.$
We write $(X^n,T^{(n)})$ for the $n$-fold product system $(X\times	\cdots \times X,T\times \cdots \times T)$.
The diagonal of $X^n$ is $\Delta_n(X)=\{(x,\ldots,x)\in X^n: x\in X\}.$
When $n=2$ we write	$\Delta(X)=\Delta_2(X)$.

\subsubsection{}
A t.d.s. $(X,T)$ is called {\em minimal} if $X$ contains no proper non-empty
closed invariant subsets. It is easy to verify that a t.d.s. is
minimal if and only if every orbit is dense. In a t.d.s. $(X, T)$ we say that a point $x\in X$ is {\em minimal} if $(\overline{O(x,T)}, T)$ is minimal.

\subsubsection{}
Let $(X,T)$ be a t.d.s. A pair $(x,y)\in X^2$ is {\it proximal} if $\inf_{n\in \Z} \rho(T^nx,T^ny)=0$; and it is {\it distal} if it is not proximal. Denote by ${\bf P}(X,T)$ the set of all proximal pairs of $(X,T)$.
A t.d.s. $(X,T)$ is called {\it distal} if $(x,x')$ is distal whenever $x,x'\in X$ are distinct.



\subsubsection{}
Let $(X,T)$ be a t.d.s. and $\mathcal{M}(X)$ be the set of all Borel
probability measures on $X$. A measure $\mu\in\mathcal{M}(X)$ is
{\em $T$-invariant} if for any Borel set $B\in\X$,
$\mu(T^{-1}B)=\mu(B)$ and we call $(X,\X,\mu,T)$ a measure-preserving system (m.p.s. for short). Denote by $\mathcal{M}(X,T)$ the set of
invariant elements of $\mathcal{M}(X)$. A measure $\mu\in\mathcal{M}(X,T)$
is {\em ergodic} if for any Borel set $B$ of $X$ satisfying
$\mu(T^{-1}B\triangle B)=0$ we have $\mu(B)=0$ or $\mu(B)=1$.
The system
$(X,T)$ is  {\em uniquely ergodic} if $\mathcal{M}(X,T)$ consists of
only one element, and it is strictly ergodic if in addition it is minimal.

Let $(X,T)$ be a t.d.s, $\mu\in\mathcal{M}(X)$. The {\em support} of $\mu$ is defined to be
$$ Supp(\mu)=\{x\in X: \text{for each  neighborhood}\ U \ \text{of} \ x,\ \mu(U)>0\}.$$
\subsection{Factor maps}\

\subsubsection{}
A {\em factor map} $\pi: X\rightarrow Y$ between two t.d.s. $(X,T)$
and $(Y, T)$ is a continuous onto map which intertwines the
actions (i.e. $\pi\circ T= T\circ \pi$); one says that $(Y, T)$ is a {\it factor} of $(X,T)$ and
that $(X,T)$ is an {\it extension} of $(Y,T)$.

\subsubsection{}
Let $\pi: (X,T)\rightarrow (Y, T)$ be a factor map. Then
$R_\pi=\{(x_1,x_2):\pi(x_1)=\pi(x_2)\}$
is a closed invariant equivalence relation, and $Y=X/ R_\pi$.
Let $(X,T)$ and $(Y,T)$ be t.d.s. and let $\pi: (X,T) \to (Y,T)$ be a factor map.
One says that:
\begin{itemize}
 \item $\pi$ is a {\it proximal} extension if
$\pi(x_1)=\pi(x_2)$ implies $(x_1,x_2) \in {\bf P} (X,T)$;
  \item $\pi$ is  a {\it distal} extension if $\pi(x_1)=\pi(x_2)$ and $x_1\neq x_2$ implies $(x_1,x_2) \not\in {\bf P} (X,T)$;
  \item $\pi$ is an {\it almost one to one} extension  if there exists a dense $G_\d$ set $X_0\subseteq X$ such that $\pi^{-1}(\{\pi(x)\})=\{x\}$ for any $x\in X_0$;
\end{itemize}

Note that an almost one to one extension is a proximal one.

\subsection{Nilsystems}\

\medskip

Let $G$ be a group. For $g, h\in G$ and $A,B \subseteq G$, we write $[g, h] =
ghg^{-1}h^{-1}$ for the commutator of $g$ and $h$ and
$[A,B]$ for the subgroup spanned by $\{[a, b] : a \in A, b\in B\}$.
The commutator subgroups $G_j$, $j\ge 1$, are defined inductively by
setting $G_1 = G$ and $G_{j+1} = [G_j ,G]$. Let $d \ge 1$ be an
integer. We say that $G$ is {\em $d$-step nilpotent} if $G_{d+1}$ is
the trivial subgroup.

\medskip
Let $d\in\N$, $G$ be a $d$-step nilpotent Lie group and $\Gamma$ be a discrete
cocompact subgroup of $G$. The compact manifold $X = G/\Gamma$ is
called a {\em $d$-step nilmanifold}. The group $G$ acts on $X$ by
left translations and we write this action as $(g, x)\mapsto gx$.
The Haar measure $\mu$ of $X$ is the unique Borel probability measure on
$X$ invariant under this action. Fix $t\in G$ and let $T$ be the
transformation $x\mapsto t x$ of $X$, i.e. $t(g\Gamma)=(tg)\Gamma$ for each $g\in G$. Then $(X, \mu, T)$ is
called a {\em $d$-step nilsystem}. In the topological setting we omit the measure
and just say that $(X,T)$ is a $d$-step nilsystem.
A {\it $d$-step pro-nilsystems} is an inverse limit of $d$-step nilsystems.

\subsection{Regionally proximal relation
of order $d$}\
\medskip

\begin{defn}\cite[Definition 3.2]{HKM} 
Let $(X, T)$ be a t.d.s. and let $d\in \N$. The points $x, y \in X$ are
said to be {\em regionally proximal of order $d$} if for any $\d  >
0$, there exist $x', y'\in X$ and a vector ${\bf n} = (n_1,\ldots ,
n_d)\in\Z^d$ such that $\rho (x, x') < \d, \rho (y, y') <\d$, and $$
\rho (T^{{\bf n}\cdot \ep}x', T^{{\bf n}\cdot \ep}y') < \d\
\text{for every  $\ep=(\ep_1,\ldots,\ep_d)\in \{0,1\}^d\setminus \{(0,0,\cdots,0)\}$},$$
where ${\bf n}\cdot \ep=\sum_{i=1}^d n_i\ep_i$.
The set of regionally proximal pairs of
order $d$ is denoted by $\RP^{[d]}$ (or by $\RP^{[d]}(X,T)$ in case of
ambiguity), and is called {\em the regionally proximal relation of
order $d$}.
\end{defn}

It is easy to see that $\RP^{[d]}$ is a closed and invariant
relation, and 
\begin{equation*}
    {\bf P}(X,T)\subseteq  \ldots \subseteq \RP^{[d+1]}\subseteq
    \RP^{[d]}\subseteq \ldots \RP^{[2]}\subseteq \RP^{[1]}.
\end{equation*}

The following theorems were proved in \cite{HKM} and \cite{SY}. 

\begin{thm}\label{thm-1}
Let $(X, T)$ be a minimal t.d.s. and let $d\in \N$. Then
\begin{enumerate}

\item $\RP^{[d]}$ is an equivalence relation.

\item $(X,T)$ is a $d$-step pro-nilsystem if and only if $\RP^{[d]}=\Delta_X$.
\end{enumerate}
\end{thm}

The regionally proximal relation of order $d$ allows us to construct the maximal $d$-step
pro-nilfactor of a system.

\begin{thm}\label{thm0}
Let $\pi: (X,T)\rightarrow (Y,T)$ be a factor map between minimal t.d.s.
and let $d\in \N$. Then,
\begin{enumerate}
  \item $\pi\times \pi (\RP^{[d]}(X,T))=\RP^{[d]}(Y,T)$.
  \item $(Y,T)$ is a $d$-step pro-nilsystem if and only if $\RP^{[d]}(X,T)\subseteq R_\pi$.
\end{enumerate}
In particular, $X_d=X/\RP^{[d]}(X,T)$, the quotient of $(X,T)$ under $\RP^{[d]}(X,T)$, is the
maximal $d$-step pro-nilfactor of $X$. 
\end{thm}

\subsection{$\infty$-step pro-nilsystems}\
\medskip

By Theorem \ref{thm-1} for any minimal t.d.s. $(X,T)$,
$\RP^{[\infty]}=\bigcap_{d\ge 1}\RP^{[d]}$
is a closed invariant equivalence relation (we write $\RP^{[\infty]}(X,T)$ in case of ambiguity). Now we formulate the
definition of $\infty$-step pro-nilsystems. 
A minimal t.d.s. $(X, T)$ is an {\em $\infty$-step pro-nilsystem} (see \cite[Definition 3.4]{D-Y}), if the equivalence
relation $\RP^{[\infty]}$ is trivial, i.e., coincides with the diagonal.

\begin{rem}
\begin{enumerate}
  \item Similar to Theorem \ref{thm0}, one can show that the quotient of a
minimal system $(X,T)$ under $\RP^{[\infty]}$ is the maximal
$\infty$-step pro-nilfactor of $(X,T)$. We denote the maximal
$\infty$-step pro-nilfactor of $(X,T)$ by $X_\infty$.

  \item A minimal system is an $\infty$-step pro-nilsystem if and only if it is
an inverse limit of minimal nilsystems \cite[Theorem 3.6]{D-Y}.

\item An $\infty$-step pro-nilsystem is distal.

\end{enumerate}
\end{rem}

\subsection{Dense sets}\
\medskip

Let $\pi:X\ra Y$ be a map between topological spaces. We say $\pi$ is {\it semi-open} if $U$ is open and non-empty in $X$,
then $\pi(U)$ has non-empty interiors; and $\pi$ is {\it open} if it sends open sets to open sets.
 It is known that if $\pi$ is a factor map between minimal systems, then $\pi$ is semi-open. Moreover, the factor map between distal systems is open.

\begin{lem}\label{veech-1} \cite[Proposition3.1]{Veech} Let $\pi:X\lra Y$ be a factor between minimal t.d.s. If $\Omega$ is a dense $G_\delta$ set of $X$,
then there is a dense $G_\delta$ set $Y_0$ of $Y$ such that for each $y\in Y_0$, $\Omega\cap \pi^{-1}(y)$
is a dense  $G_\delta$ set of $\pi^{-1}(y).$
\end{lem}

\section{The proof of Theorem A}
In this section we will show Theorem A. As we said before, Theorem A is proved by using Theorem HSY, an ergodic argument to deal with the distal case, and a method to put them together. We start with some preparation.

\subsection{Preparation}
Note that we say polynomials $p$ and $q$ are {\it distinct} if $p-q$ is not constant (it is called {\it essentially distinct} in \cite{Lei05-1}). Let $C=\{p_1,\ldots, p_d\}$ be a set of distinct non-constant integral polynomials vanishing at $0$. 
Given a factor map $\pi: (X,T)\rightarrow (Y,T)$ and $d\ge 2$,
the t.d.s. $(Y,T)$ is said to be a {\em  TCF of order d for $C$ 
of $(X,T)$}, if there exists a dense $G_\d$
subset $\Omega$ of $X$ such that for each $x\in \Omega$ the orbit
closure $L_x^{C}=\overline{\{(T^{p_1(n)}x, \ldots, T^{p_d(n)}x): n\in\Z\}}$ is $\pi^{(d)}=\pi\times \cdots \times
\pi$ ($d$-times) saturated. The following was proved by in \cite{HSY-22-2}.

\begin{thm}\label{polsaturation}
Let $(X,T)$ be a minimal t.d.s., and $\pi:X\rightarrow X_\infty$ be the factor map from $X$ to its maximal $\infty$-step pro-nilfactor $X_\infty$
of $X$. Then there are minimal t.d.s. $X^*$ and $X_\infty^*$ which are almost one to one
extensions of $X$ and $X_\infty$ respectively, an open factor map $\pi^*$ and a commuting diagram below
\[
\begin{CD}
X @<{\varsigma^*}<< X^*\\
@VV{\pi}V      @VV{\pi^*}V\\
X_\infty @<{\varsigma}<< X_\infty^*
\end{CD}
\]
such that there is a $T$-invariant dense $G_{\delta}$ subset $X^*_0$ of $X^*$ having the following property:  for all $x\in X^*_0$, for any non-empty open subsets $V_1,\ldots, V_d$ of $X^*$ with $\pi(x)\in \bigcap _{i=1}^d \pi^*(V_i)$ and distinct non-constant integral polynomials $p_1,p_2,\ldots, p_d$ with $p_{i}(0)=0$, $i=1,2,\ldots, d$, there is some $n\in \N$ such that
$$x\in  T^{-p_1(n)}V_1\cap T^{-p_2(n)}V_2\cap \cdots \cap T^{-p_d(n)}V_d.$$
\end{thm}

Now we give the following lemma to explain the equivalence between Theorem \ref{polsaturation} and Theorem HSY.
\begin{lem}Theorem \ref{polsaturation} and Theorem HSY are equivalent.
\end{lem}
\begin{proof} It is clear that Theorem HSY implies Theorem \ref{polsaturation}, as the maximal $\infty$-step pro-nilfactor of $X^*$ is also
$X_\infty$, see for example \cite{5p}.

To show the other implication, let us assume  
the commuting diagram in Theorem \ref{polsaturation}. As $X^*$ and $X_\infty^*$
are almost one to one extensions of $X$ and $X_\infty$ respectively,
$$\Omega_1=\{y\in X_\infty: |\varsigma^{-1}(y)=1|\}\  \text{and}\  \Omega_2=\{y\in X: |\varsigma^{*-1}(y)=1|\}$$
are dense $G_\delta$ sets. Let $X^{*}_0$ be the dense $G_\d$ subset of $X^*$ in Theorem \ref{polsaturation}, and set $\Omega=\pi^{-1}(\Omega_1)\cap\Omega_2\cap\varsigma^*(X^{*}_0)$, which is a dense $G_\delta$ subset of $X$ (as $\varsigma^* $ is semi-open, $\varsigma^*(X_0^*)$ is also a dense $G_\d$ subset of $X$).

Now let $x\in \Omega$, and $x_1,\cdots,x_d\in\pi^{-1}\pi(x).$ Since $\varsigma\pi^*=\pi\varsigma^*$, we get that
$$(\pi^*)^{-1}\varsigma^{-1}\pi(x)=(\varsigma^*)^{-1}\pi^{-1}\pi(x).$$
Let $\{y\}=(\varsigma^*)^{-1}(x)$. Note that since $\pi(x)\in \Omega_1$, we have $|\varsigma^{-1}\pi(x)|=1$
and in fact, $\varsigma^{-1}\pi(x)=\{\pi^*y\}$. So there are $y_1,\ldots,y_d\in (\pi^*)^{-1}\pi^*(y)$ with $\varsigma^*(y_i)=x_i, 1\le i\le d$.

By Theorem \ref{polsaturation}, $(y_1,\ldots, y_d)\in \overline{\{(T^{p_1(n)}y, \ldots, T^{p_d(n)}y): n\in\Z\}}$ which implies that
$$(x_1,\ldots, x_d)\in \overline{\{(T^{p_1(n)}x, \ldots, T^{p_d(n)}x): n\in\Z\}}.$$
\end{proof}

Let $G$ be a locally compact group, written with multiplicative notation, and let $m_G$ be its Haar measure. A {\em F{\o}lner sequence} in $G$ is a sequence $\bm{\Phi}=(\Phi_N)_{N\in\mathbb{N}}$ of compact subsets of $G$ each of which has positive Haar measure and such that for any $g\in G$,
$$\frac{m_{G}(g\Phi_{N}\Delta\Phi_N)}{m_G(\Phi_N)}\to 0, \ as \ N\to\infty,$$
where $\Delta$ denotes the symmetric difference and $g\Phi_N=\{gh:h\in\Phi_N\}.$

\medskip

For the definition of the semi-norm of order $k$ see \cite{HK05}. The convergence of the multiple ergodic averages in $L^2$ was studied in
\cite{HK05-1, Lei05-1}. We need a result from \cite[Theorem 3]{Lei05-1}.
\begin{thm}\label{characteristic}
Let $(X,\mathcal{X},\mu,T)$ be an ergodic system. For any $r, b\in \N$ there exists $k\in\N$ such that for any system of
distinct non-constant  polynomials $p_1,\ldots,p_r: \Z^d\lra \Z$ of degree $\le b$
and any $f_1,\ldots,f_r\in L^\infty$ with $|||f_1|||_k=0$, one has
$$\lim_{N\ra \infty} \frac{1}{|\Phi_N|}\sum_{u\in \Phi_N}T^{p_1(u)}f_1\cdots T^{p_r(u)}f_r = 0$$
in $L^2$ for any F{\o}lner sequence $\Phi_N$ in $\Z^d$.
\end{thm}

For an ergodic system $(X,\mathcal{X},\mu,T)$, the authors in \cite{HK05} showed that it admits a maximal $k$-step pro-nilfcator, denoted by $Z_k$. We remark that for an $\infty$-step minimal pro-nilsystem, $X_k=Z_k$, where the measure is the unique measure on $Z_k$.

Let $(Z_{k},\mathcal{Z}_{k},\mu_k,T)$ be the maximal $k$-step pro-nilfcator of an ergodic system $(X,\mathcal{X},\mu,T)$.
It is known in \cite{HK05} that for $f\in L^\infty$, $\mathbb{E}(f|\mathcal{Z}_{k-1})=0\Leftrightarrow |||f|||_k=0,$
 where $\mathbb{E}(f|\mathcal{Z}_{k-1})$ is the conditional expectation. For $f\in L^\infty$, write $f=(f-\mathbb{E}(f|\mathcal{Z}_{k-1})+\mathbb{E}(f|\mathcal{Z}_{k-1})$. It is easy to see that
$\mathbb{E}((f-\mathbb{E}(f|\mathcal{Z}_{k-1}))|Z_{k-1})=0$.
So in Theorem \ref{characteristic}, we can replace $f_i$ by $\mathbb{E}(f_i|\mathcal{Z}_{k-1})$.
Thus, we can rewrite the conclusion of Theorem \ref{characteristic} as:
\begin{equation}\label{constant-k}
\vert\vert\lim_{N\to\infty}\frac{1}{N}\sum_{n=0}^{N-1}(\prod_{i=0}^{d}T^{p_{i}(n)}f_i-\prod_{i=0}^{d}T^{p_{i}(n)}
\mathbb{E}(f_i|\mathcal{Z}_{k}))\vert\vert_{L^2(\mu)}= 0,
\end{equation}
for some $k$ only depending on the the polynomials $p_1,\cdots,p_d$.

\medskip
The other result we need is the following theorem from \cite[Theorem A]{BL96}.
\begin{thm}\label{polbl}
Let $(X,\mathcal{X},\mu,T)$ be a measure preserving transformation and let $A\in\mathcal{X}$ be a set with positive measure, then for each $C=\{p_1,\ldots, p_k\}$ of distinct non-constant integral polynomials vanishing at 0,
$$\liminf_{N\to\infty}\frac{1}{N}\sum_{n=0}^{N-1}\mu(A\cap T^{-p_1(n)}A\cap T^{-p_2(n)}A\cap\cdots \cap T^{-p_{k}(n)}A)>0.$$
\end{thm}

\subsection{The distal case}\label{distalcase}
To show the distal case we need a lemma from \cite{4p}.

\begin{lem}\label{fib}\cite[Proposition 5.5]{4p}
Let $(X,T)$ be a strictly ergodic system with a unique invariant measure $\mu$ and let $\pi:(X,T)\to (Y,T)$ be a distal extension. Let $\mu=\int_{Y}\mu_{y}d\nu$ be the disintegration of $\mu$ over $\nu$, where $\nu=\pi_{*}(\mu)$. Then there is $Y_0\subset Y$ with full measure such that for each $y\in Y_0$, $supp(\mu_y)=\pi^{-1}(y)$.
\end{lem}

Now we are able to show

\begin{thm}\label{distal}
Let $(X,T)$ be a minimal distal system and $C=\{p_1,\ldots,p_d\}$ be a collection of distinct non-constant integral polynomials vanishing at $0$. Then there is $k\in\N$ (only depending on $C$) such that $(X_k,T)$ is the TCF of order $d$ for $C$ of $(X,T)$.
\end{thm}

\begin{proof} First we know that $(X_{\infty},T)$ is the TCF of order $d$ for $C$ of $(X,T)$ by Theorem HSY, since the factor
map $X\ra X_\infty$ is open. It remains to show $(X_k,T)$ is TCF of order $d$ for $C$ of $(X_{\infty},T)$, as the composition of two factor maps
with TCF keeps the TCF property (one can check this directly, or see it from the proof of Theorem A).

\medskip
Let $\pi_k:X_\infty\to X_k$ be the factor map, and $k$ be the number defined in (\ref{constant-k}). Set $\mu$ be the uniquely ergodic measure of $(X_{\infty},T)$, and $(Z_k,\mathcal{Z}_{k},\mu_k,T)$ be the maximal $k$-step pro-nilfactor of $(X_{\infty},\mathcal{X}_\infty,\mu,T)$. It is known that
$Z_k=X_k$, and so $\pi_k$ can be viewed as the continuous factor map from $(X_\infty,\mu,T)$ to $(Z_k,\mu_k,T)$.

Let $\mu=\int_{Z_k}\mu_{z}d\mu_{k}(z)$ be the disintegration of $\mu$ over $\mu_k$. For $d\in\N$, let $L_{d,k}^{\mu}={\rm Supp}(\lambda_{d}^{k})$, where
$$\lambda_{d}^{k}=\int_{Z_k}\prod_{j=1}^d\mu_{y}\ d\mu_{k}(y).$$

Similar to the argument in \cite[Theorem 5.9]{4p},
we now show that ${\rm Supp}(\lambda_{d}^{k})=R_{\pi_k}^d$. First we note that $\lambda_d^k(R_{\pi_k}^d)=1$, so
${\rm Supp}(\lambda_{d}^{k})\subset R_{\pi_k}^d$. By Lemma \ref{fib}  there is a measurable set
$Y_0\subset X_{d}$ with full measure such that for any $y\in Y_0$, ${\rm Supp}(\mu_y)={\pi_k}^{-1}(y)$. Let $W= {\rm Supp}(\lambda_{d}^{k})$. Since
$$\lambda_{d}^{k}(W)=\int_{Y_0}\prod_{j=1}^d\mu_{y}(W)\ d\mu_k(y)=1,$$
we have that for a.e. $y\in Y$, $\prod_{j=1}^d\mu_{y}(W)=1$. This implies that
$\prod_{j=1}^d{\rm Supp}(\mu_y)\subset W$, a.e. $y\in Y$. Thus by the distality of $\pi_k$, the map $y\mapsto \pi_k^{-1}(y)$ (from $X_k$ to $2^{X_\infty}$) is continuous and we conclude that $R_{\pi_k}^d\subset{\rm Supp}(\lambda_{d}^{k})$. Thus, we get that ${\rm Supp}(\lambda_{d}^{k})=R_{\pi_k}^d$.

\medskip

Now let $C=\{p_1,\ldots,p_d\}$ be the polynomials mentioned in the theorem. We now show a claim:

\noindent$\mathbf{Claim}$: 
For any non-empty open subsets $U_0,\cdots,U_d$ of $X_\infty$ with $\bigcap_{i=0}^{d}\pi_k(U_i)\neq \emptyset$, 
$$U_0\cap T^{-p_{1}(n)}U_1\cap T^{-p_{2}(n)}U_2\cap\cdots\cap T^{-p_{d}(n)}U_d\neq \emptyset.$$

\begin{proof}
By (\ref{constant-k}) we have $$\vert\vert\lim_{N\to\infty}\frac{1}{N}\sum_{n=0}^{N-1}(\prod_{i=1}^{d}T^{p_{i}(n)}f_i-\prod_{i=1}^{d}T^{p_{i}(n)}\mathbb{E}_{\mu}(f_i|\mathcal{Z}_{k}))\vert\vert_{L^2(\mu)}= 0.$$
Let $(x_0,x_1,\cdots,x_d)\in R_{\pi_k}^{d+1}={\rm Supp}(\lambda_{d+1}^{k})$, and let $(U_0,U_1,\cdots,U_d)$ be a neighborhood of $(x_0,x_1,\cdots,x_d)$,
then it follows that 
\begin{equation}\label{conexp}
\lambda_{k}(U_0\times U_1\times\cdots\times U_d)=\int_{Z_{k}}\E(1_{U_0}|\mathcal{Z}_{k})\E(1_{U_1}|\mathcal{Z}_{k})\cdots \E(1_{U_d}|\mathcal{Z}_{k}) d\mu_{{k}}>0.
\end{equation}

By (\ref{constant-k}), we have
\begin{equation*}
\begin{split}
     &\lim_{N\to\infty}\frac{1}{N}\sum_{n=0}^{N-1}\mu(U_0\cap T^{-p_1(n)}U_1\cap T^{-p_2(n)}U_2\cap\cdots\cap T^{-p_{d}(n)}U_d)\\
       = & \lim_{N\to\infty}\frac{1}{N}\sum_{n=0}^{N-1}\int_{X}1_{U_0}(x)1_{U_1}(T^{p_1(n)}x)1_{U_2}(T^{p_2(n)}x)\cdots1_{U_d}(T^{p_{d}(n)}x) d\mu(x)\\
       = &\lim_{N\to\infty}\frac{1}{N}\sum_{n=0}^{N-1}\int_{Z_{k}}\mathbb{E}(1_{U_0}|\mathcal{Z}_{k})(z)
      \mathbb{E}(1_{U_1}|\mathcal{Z}_{k})(T^{p_{1}(n)}z)\cdots\mathbb{E}(1_{U_d}|\mathcal{Z}_{k})(T^{p_{d}(n)})d\mu_{k}(z)\\
     \geq &\liminf_{N\to\infty}\frac{1}{N}\sum_{n=0}^{N-1}a^{d+1}\int_{Z_k}1_{A_a}(z)1_{A_a}(T^{p_1(n)}z)\cdots1_{A_a}(T^{p_{d}}(n)z)d\mu_{k}(z)\\
      =&\liminf_{N\to\infty}\frac{1}{N}\sum_{n=0}^{N-1}a^{d+1}\mu_k(A_a\cap T^{-p_1(n)}A_a\cap T^{-p_2(n)}A_a\cap\cdots\cap T^{-p_{d}(n)}A_a),
\end{split}
\end{equation*}
where $a>0$ and
$$A_a=\{z\in Z_{k}:\mathbb{E}(1_{U_0}|\mathcal{Z}_{k})(z)>a, \mathbb{E}(1_{U_1}|\mathcal{Z}_{k})(z)>a,\cdots,\mathbb{E}(1_{U_d}|\mathcal{Z}_{k})(z)>a\}.$$

As $\mathbb{E}(1_{U_i}|\mathcal{Z}_{k})\leq1$ , $i=0,1,\cdots, d$, by (\ref{conexp}) we can get that
\begin{equation*}
\begin{split}
0<b&:=\int_{Z_k}\mathbb{E}(1_{U_0}|\mathcal{Z}_{k})\mathbb{E}(1_{U_1}|\mathcal{Z}_{k})\cdots \mathbb{E}(1_{U_d}|\mathcal{Z}_{k})d\mu_{k}\\
&=\int_{A_a}\mathbb{E}(1_{U_0}|\mathcal{Z}_{k})\cdots\mathbb{E}(1_{U_d}|\mathcal{Z}_{k})d\mu_{k}+\int_{Z_{k}\setminus A_a}\mathbb{E}(1_{U_0}|\mathcal{Z}_{k})\cdots\mathbb{E}(1_{U_d}|\mathcal{Z}_{k})d\mu_{k}\\
&\leq\mu_{k}(A_a)+a\mu_{k}(Z_{k}\setminus A_a)=a+(1-a)\mu_{k}(A_a),\\
\end{split}
\end{equation*}
here we use that fact that for $z\in Z_{k}\setminus A_a$, $\mathbb{E}(1_{U_d}|\mathcal{Z}_{k})(z)\le a$. Hence there exists $a>0$ such that $\mu_{k}(A_a)>0$. 
So by Theorem \ref{polbl}
$$\lim_{N\to\infty}\frac{1}{N}\sum_{n=0}^{N-1}\mu(U_0\cap T^{-p_1(n)}U_1\cap\cdots\cap T^{-p_{d}(n)}U_d)>0.$$
Then there exists $n\in\mathbb{N}$ such that
$$U_0\cap T^{-p_{1}(n)}U_1\cap T^{-p_{2}(n)}U_2\cap\cdots\cap T^{-p_{d}(n)}U_d\neq\emptyset.$$
This ends the proof of the claim. 
\end{proof}

Now we show that $(X_k,T)$ is the TCF of order $d$ for $C$ of $(X_\infty,T)$ using the claim we just proved. This is exact the proof of  \cite[Theorem 3.6]{HSY-22-2}. This ends the proof.
\end{proof}

\subsection{Proof of Theorem A}
In this subsection we show Theorem A. For a compact metric space $X$, we use $2^X$ to denote the set of all non-empty closed subsets of $X$ (with Hausdorff topology). We need a lemma which can be verified directly by results in \cite{C}.
\begin{lem}\label{easy}
Let $\pi: (X,T) \ra (Y,T)$ be a factor map between minimal t.d.s. and let $\Pi: Y \ra 2^X$ be  defined by $y\mapsto \pi^{-1}(y)$. Then we have
$\{y\in Y: \Pi\  \text{is continuous at}\ y\}$ is a dense $G_\delta$ set of $Y$.
\end{lem}

Another lemma we need is a direct consequence of Theorem HSY
\begin{lem}\label{dir-hsy}
Let $d\in \N$, $(X,T)$ be a minimal system, and $\pi_1:X\rightarrow X_\infty$ be the factor map. Then there is a dense $G_\delta$ set $\Omega_1$ of $X$ such that for any $x\in \Omega_1$, any distinct non-constant integral polynomials $p_1, p_2,\ldots, p_d$ vanishing at $0$ and any $m\in Z$
$$\prod_{i=1}^d\pi_1^{-1}\pi_1(T^{p_i(m)}x)\subset \overline{\{(T^{p_1(n)}x, \ldots, T^{p_d(n)}x): n\in\Z\}}.$$
\end{lem}
\begin{proof} Since the set of integral polynomials is countable, we may assume that the set $\Omega$ in Theorem HSY is satisfied for any
distinct non-constant integral polynomials vanishing at $0$. Set $\Omega_1$ to be the set. For $x\in \Omega_1$ and $m\in\Z$, let
$x_i\in \pi_1^{-1}\pi_1(T^{p_i(m)}x)$, and $V_i$ be a neighborhood of $x_i$, $1\le i\le d$. It is clear that $T^{-p_i(m)}V_i\cap \pi_1^{-1}\pi_1(x)
\not=\emptyset$, $1\le i\le d$. Applying Theorem HSY to polynomials $\{p_1(n+m)-p_1(m), \ldots, p_d(n+m)-p_d(m)\}$ we get the result.
\end{proof}

Now we are ready to show Theorem A.

\begin{proof}[Proof of Theorem A]
Let $C=\{p_1,\ldots,p_d\}$ be distinct non-constant integral polynomials with $p_i(0)=0, 1\le i\le d$.
Let $\pi_1:X\to X_{\infty}$ and $\pi_2: X_\infty\ra X_k$ be the factor maps, where $k$ is the number defined in Theorem \ref{distal}.

Let $\pi_k=\pi_2\pi_1:X\ra X_k$ and $\Omega_1$ be the set defined in Lemma \ref{dir-hsy}.

By Theorem \ref{distal}, there is a dense $G_\delta$ set $\Omega_2$ of $X_\infty$ such that for any 
$y\in \Omega_2$,
\begin{equation}\label{eq-4-2}
(\pi_2^{-1}\pi_2(y))^d\subset \overline{\{(T^{p_1(n)}y, \ldots, T^{p_d(n)}y): n\in\Z\}}.
\end{equation}
Let $\Omega_3$ be the dense $G_\delta$ subset of $X_\infty$ such that if $y\in \Omega_3$, then $y$ is a continuous point of $y\mapsto \pi_1^{-1}(y)$
(see Lemma \ref{easy}).

Let $\Omega'=\Omega_1\cap \pi_1^{-1}(\Omega_2\cap \Omega_3)$. Then $\Omega'$ is a dense $G_\delta$ subset of $X$. Then there is a dense $G_\delta$ set
$\Omega_4$ of $X_k$ such that for each $y\in \Omega_4$, $\pi_k^{-1}(y)\cap \Omega'$ is a dense $G_\delta$ set of $\pi_k^{-1}(y)$ (see Lemma \ref{veech-1}). Recall that $\pi_k=\pi_2\pi_1:X\ra X_k$.

\medskip
Set $\Omega=\pi_k^{-1}\Omega_4\cap \Omega'$. We claim $\Omega$ is the set we need. To show the claim let $x\in \Omega$ and
$$x_1,\ldots,x_{d}\in (\pi_k^{-1}\pi_k(x))\cap \Omega'=(\pi_1^{-1}\pi_2^{-1}\pi_2\pi_1(x))\cap \Omega'.$$

Then
$\pi_1(x_1),\ldots,\pi_1(x_{d})\in \pi_2^{-1}\pi_2\pi_1(x).$
It is clear that $\pi_2(\pi_1(x_i))=\pi_2\pi_1(x)$, $1\le i\le d$  and $\pi_1(x)\in \Omega_2\subset X_\infty$. Thus by (\ref{eq-4-2}) there is a sequence $\{n_i\}$ of $\Z$ such that
$$(T^{p_1(n_i)}\pi_1(x),\ldots, T^{p_d(n_i)} \pi_1(x))\ra (\pi_1(x_1),\ldots,\pi_1(x_{d})),\ as \ i\to\infty.$$
That is,
$$(\pi_1(T^{p_1(n_i)}x),\ldots, \pi_1  (T^{p_d(n_i)}x))\ra (\pi_1(x_1),\ldots,\pi_1(x_{d})),\ as \ i\to\infty.$$


Since $x_i\in \Omega'$ we have that $\pi_1(x_i)\in \Omega_3$,  which implies that $\pi_1(x_i)$ is a continuous point for
$y\mapsto \pi_1^{-1}(y)$, $1\le i\le d$. By Lemma \ref{dir-hsy} we know that for any $i\in \N$
$$\prod_{j=1}^d\pi_1^{-1}\pi_1(T^{p_j(n_i)}x)\subset \overline{\{(T^{p_1(n)}x, \ldots, T^{p_d(n)}x): n\in\Z\}}.$$ So, we get that
$(x_1,\ldots,x_{d})\in \overline{\{(T^{p_1(n)}x, \ldots, T^{p_d(n)}x): n\in\Z\}}.$
Since such tuples $(x_1,\ldots, x_d)$ are dense in $\pi_k^{-1}\pi_k(x)$, we conclude that
$$(\pi_k^{-1}\pi_k(x))^d\subset \overline{\{(T^{p_1(n)}x, \ldots, T^{p_d(n)}x): n\in\Z\}}.$$
This ends the proof.
\end{proof}

\subsection{Remarks} It is clear that if $C$ is a set of linear polynomials, the number $k$ can be chosen to be $d-1$ by
\cite{HK05} or \cite{Zie}.

\medskip
Let $(X,\X,\mu, T)$ be a m.p.s. Let
$${\mathcal H}_{\rm rat}=\overline{\{f\in L^2(X,\mu): \exists a \in \N \ s.t.\ T^a f=f\}},$$
and let $\X_{\rm rat}\subseteq \X$ such that
${\mathcal H}_{\rm rat}=L^2(X,\X_{\rm rat},\mu).$

\begin{thm}\cite[Lemma 3.14]{F}\label{Furstenberg}
Let $(X,\X,\mu, T)$ be a m.p.s. and let $q(n)$ be a non-constant integral polynomial. Then for all $f\in L^2(X, \mu)$,
\begin{equation*}
\frac{1}{N}\sum_{n=0}^{N-1}T^{q(n)}f- \frac{1}{N}\sum_{n=0}^{N-1}T^{q(n)} \E(f|\X_{\rm rat})\longrightarrow 0, \ N\to\infty ,
\end{equation*}
in $L^2(X,\mu)$.
In particular, if $(X,\X,\mu,T)$ is totally ergodic \footnote{A m.p.s. $(X,\X,\mu,T)$ is totally ergodic if $(X,\X,\mu,T^k)$ is ergodic for all $k\in \Z\setminus\{0\}$.}, then
\begin{equation*}
\frac{1}{N}\sum_{n=0}^{N-1}T^{q(n)}f \longrightarrow \int_X f d\mu, \ N\to\infty ,
\end{equation*}
in $L^2(X,\mu)$.
\end{thm}

For a minimal system $(X,T)$ we may define $X_{rat}$, a topological version of $Z_{rat}$ (where $(Z_{rat}, \X_{\rm rat},\mu_{rat},T)$ is the factor), which is called {\it rational topological Kronecker factor} in \cite{GKR}. It is an adding machine, and maximal in the sense that if $Y$ is an adding machine and is a factor of $X$ then $Y$ is a factor of $X_{rat}$. $X_{rat}$ reflects the number of minimal subsystems of $(X,T^n)$, $n\in\Z\setminus\{0\}$. Note that for an $\infty$-step pronilsystem,  $X_{rat}=Z_{rat}$.

By using Lemma \ref{dir-hsy}, the arguments in Subsection \ref{distalcase} (replacing Theorem \ref{characteristic} by Theorem \ref{Furstenberg}) we get that if $(X,T)$ is minimal and $p$ is a non-constant integral polynomial vanishing at $0$, then there is a dense $G_\delta$ set $\Omega$ of $X$ such that for each $x\in \Omega$, $\pi_{rat}^{-1}\pi_{rat}(x)\subset \overline{\{T^{p(n)}x: n\in\Z\}}$, where $\pi_{rat}: X\ra X_{rat}$ is the factor map.

Particularly, when $(X,T)$ is totally minimal \footnote{$(X,T)$ is totally minimal if $(X,T^n)$ is minimal for any $n\in\Z\setminus\{0\}$.}
$X_{rat}$ is trivial. Thus, the above conclusion gives another approach to the result in \cite{5p, Qiu} for a single polynomial.

\section{Applications}
In this section we will provide applications of Theorem A, i.e. prove Theorem B. Namely, we will use it to study the regionally proximal relation along polynomials. We begin with the general case and then obtain a corollary when the polynomials are linear.
\subsection{General cases} Let us recall a definition.
Let $(X,T)$ be a t.d.s., $d\in\N$ and $C=\{p_1,\ldots,p_d\}$ be distinct non-constant integral polynomials with $p_i(0)=0$, $1\le i\le d$. We say $(x,y)$ is regionally proximal along $C$ if for each neighborhood $U\times V$ of $(x,y)$ and $\ep>0$ there are $(x',y')\in U\times V$ and $n\in\Z$
such that
$$\rho(T^{p_i(n)}x',T^{p_i(n)}y')<\ep,\ 1\le i\le d.$$
The set of all such pairs is denoted by $\RP^{[d]}_{C}(X,T)$. 

\medskip
Then it's clear that:
\begin{prop}\label{RPC} Let $(X,T)$ be a t.d.s., and $\pi:X\ra Y$ be a factor map. Then we have
\begin{enumerate}
\item $\RP^{[d]}_{C}$ is a closed $T\times T$-invariant relation.

\item $(\pi\times \pi)(\RP^{[d]}_{C}(X))\subset \RP^{[d]}_{C}(Y).$

\end{enumerate}
\end{prop}

For a t.d.s. $(X,T)$ we use $h(T)$ to denote the topological entropy of $T$. To study the entropy property of the minimal systems with $\RP^{[d]}_{C}(X)=\Delta$ we need the properties of  {\it an entropy pair} which was studied in \cite{BL}. It is known that for a t.d.s. $(X,T)$, $h(T)>0$ if and only if there is an entropy pair
$(x_1,x_2)$ with $x_1\not=x_2$.
It was shown in \cite{HY06} that an entropy pair has {\it an independence set}
of positive density \footnote{$F\subset \N$ has positive density if $\lim_{N\ra \infty} \frac{1}{N}|F\cap \{1,\ldots, N\}|>0$.}. Note that for $F\subset \N$ we say $(x_1,x_2)$ has an independence set $F$  if for any finite $S \subset F$ and any neighborhood $U_1\times U_2$ of $(x_1,x_2)$ we have
$$T^{-s_1}U_{i_1}\cap \cdots  \cap T^{-s_l}U_{i_l}\not=\emptyset,$$
for any $(i_1,\ldots,i_l)\in \{1,2\}^l,$ if $S=\{s_1,\ldots,s_l\}$.

\medskip
Now we are ready to show Theorem B.
\begin{proof}[Proof of Theorem B]
Let $l$ be an integral polynomial with  $l(0)=0$ such that $l, p_1,\ldots,p_d,l+p_1,\ldots,l+p_d$ are distinct non-constant polynomials. For example,
any integral polynomial $l$ with  $l(0)=0$ and $deg(l)>\max\{deg(p_1), \ldots, deg(p_d)\}$ is the one we need.
Set $$q_0=l, q_1=p_1,q_2=p_2, \ldots, q_d=p_d, q_{d+1}=l+p_1, \ldots, q_{2d}=l+p_d, $$ and
$C_{2d+1}=\{q_0, q_1,\ldots, q_{2d}\}$.

Then by the Theorem A, there exists a $k\in\N$ depending only on $C_{2d+1}$ such that for a residual subset $\Omega$ of $X$, and $x\in \Omega$
we have $(\pi^{-1}_{k}(\pi_{k}(x)))^{2d+1}\subset L_x^{C_{2d+1}}$, where $\pi_{k}: X\to X_{k}$ is the factor map.

Now let $x_1\in \Omega$, and  $x_2\in \pi_k^{-1}(\pi_k(x_1))$. Let $\ep>0$ and $U_1,U_2$ be neighborhoods of $x_1,x_2$ respectively with $\diam(U_2)<\ep$.
Since $(\pi^{-1}_{k}(\pi_{k}(x_1)))^{2d+1}\subset L_{x_1}^{C_{2d+1}}$ we get that there is $n\in\Z$ such that
$$T^{q_0(n)}x_1, T^{q_1(n)}x_1,\ldots, T^{q_{2d}(n)}x_1\in U_2.$$ That is, we have
$$T^{l(n)}x_1, T^{p_1(n)}x_1, \ldots, T^{p_d(n)}x_1, T^{l(n)+p_1(n)}x_1, T^{l(n)+p_2(n)}x_1\ldots, T^{l(n)+p_d(n)}x_1\in U_2.$$

Set $x'=x_1\in U_1,y'=T^{l(n)}x_1\in U_2$. Then we have
$$\rho(T^{p_i(n)}x', T^{p_i(n)}y')<\ep, \ 1\le i\le d.$$

This implies that $(x_1,x_2)\in \RP^{[d]}_C$, where $C=\{p_1,\ldots,p_d\}$. As $\RP^{[d]}_C=\Delta$ we conclude that $x_2=x_1$, i.e.
$\pi_k$ is almost one-to-one.

\medskip
Particularly, if $(X,T)$ is distal then it is a $k$-step pro-nilsystem by what we just proved.

\medskip
Now we show $h(T)=0$. Assume the contrary that $h(T)>0$. So, there is an entropy pair $(x_1,x_2)$ in $X\times X$ with $x_1\not=x_2$. Fix $\ep>0$. For any neighborhood $U_1\times U_2$ of $(x_1,x_2)$ with $\diam(U_2)<\ep$, there is an independent set $F$ of positive density associated to $U_1\times U_2$.
For the polynomials $l, l+p_1,\ldots,l+p_d$, by Bergelson and Leibman's theorem (see \cite{BL96})
there are $m,n\in \Z$ with
$$m+l(n),m+l(n)+p_1(n), m+l(n)+p_2(n), \ldots, m+l(n)+p_d(n)\in F.$$ Thus, we have
$$T^{-m-l(n)}U_1\cap T^{-m-l(n)-p_1(n)}U_2\cap \cdots \cap T^{-m-l(n)-p_d(n)}U_2\not=\emptyset$$
and $$T^{-m-l(n)}U_2\cap T^{-m-l(n)-p_1(n)}U_2\cap \cdots \cap T^{-m-l(n)-p_d(n)}U_2\not=\emptyset.$$
This implies that
$$U_1\cap T^{-p_1(n)}U_2\cap \cdots \cap T^{-p_d(n)}U_2\not=\emptyset$$
and $$U_2\cap T^{-p_1(n)}U_2\cap \cdots \cap T^{-p_d(n)}U_2\not=\emptyset.$$
Thus there are $y_1\in U_1$ with $T^{p_i(n)}y_1\in U_2$, and $y_2\in U_2$ with $T^{p_i(n)}y_2\in U_2$, $1\le i\le d$.
This implies that $\rho(T^{p_i(n)}y_1, T^{p_i(n)}y_2)<\ep$, $1\le i\le d$. Since $\ep$ is arbitrary, we get that
$(x_1,x_2)\in \RP^{[d]}_{C}$, a contradiction. This proves that $h(T)=0$.
\end{proof}

\subsection{Linear case}
To refine Theorem B for the linear case, we need a lemma from \cite[Theorem 3.8]{4p}.
\begin{lem}\label{AP}
Let $\pi:(X,T)\lra (Y,T)$ be a proximal extension between two t.d.s.
Then $\AP^{[d]}(X)\supset R_\pi=\{(x,y)\in X^2:\pi(x)=\pi(y)\}$ for any $d\in\N$.
\end{lem}
Now let $C$ be the polynomials $\{n,2n,\cdots,dn\}$. We have a corollary of Theorem B.
\begin{cor}\label{cor-impot}Let $(X,T)$ be a minimal t.d.s. with $\AP^{[d]}=\Delta$
for some $d\in \N$, then it is a $d$-step pro-nilsystem.
\end{cor}
\begin{proof} As $\AP^{[d]}=\Delta$,  Theorem B tells us that
$(X,T)$ is an almost one-to-one extension of some $X_{k}$ for some $k\in\N$.

By Lemma \ref{AP}, we conclude that $R_\pi=\Delta$ (where $\pi:X\ra X_k$ is the factor map) and thus $(X,T)$ is a $k$-step pro-nilsystem. Then we have $\AP^{[d]}=\RP^{[d]}=\Delta$ (see \cite[Theorem 5.9]{4p}) which implies that $(X,T)$ is a $d$-step pro-nilsystem by Theorem \ref{thm0}.
\end{proof}

We note that Corollary \ref{cor-impot} gives a negative answer to a conjecture in \cite[Conjecture 4]{5p}, which states that there is a minimal
t.d.s. $(X,T)$ such that $\AP^{[2]}=\Delta$, and $(X,T)$ is not distal. 

\subsection{Questions} We end the paper by stating some open questions.

Recently, Qiu and Yu \cite{QiuYu} obtained a measure-theoretical version of the saturation theorem along cubes.
We believe that Theorem A should have a measure-theoretical version similar to \cite[Theorem A]{QiuYu} for polynomials (then we could conclude that
$\RP^{[d]}_{C}(X,T)=\Delta$ implies that $(X,T)$ is uniquely ergodic). Unfortunately, we do not
know how to prove it even the polynomials are linear at this moment.

\medskip
Another question is that: is there a minimal t.d.s. $(X,T)$ such that $\RP^{[1]}_{\{n^2\}}=\Delta$ and $(X,T)$ is not an equicontinuous system?
We think that such a system exists.

\end{document}